\documentclass{birkjour}
\usepackage{pgf}
\usepackage{tikz}
\usetikzlibrary{arrows,automata}
\usepackage{circuitikz}
\newtheorem{example}{Example}[section]
\newtheorem{theorem}{Theorem}[section]

\newtheorem{thm}{Theorem}[section]
\newcommand{\be}{\begin{eqnarray}}
\newcommand{\ee}{\end{eqnarray}}
\newcommand{\bd}{\begin{displaymath}}
\newcommand{\ed}{\end{displaymath}}
\newcommand {\bad}{\begin{aligned}}
\newcommand {\ead}{\end{aligned}}

\newtheorem{n.b}[thm]{Note}
\newtheorem{lem}[thm]{Lemma}

\newtheorem{prop}[thm]{Proposition}

\newtheorem{cor}[thm]{Corollary}
\newcommand{\bal}{\mbox{\boldmath $\alpha$}}

\newcommand{\bmi}{\mbox{\footnotesize\boldmath $\mu$}}
\newcommand{\bmu}{\mbox{\boldmath $\mu$}}

\newcommand {\bdm}{\begin{displaymath}}
\newcommand {\edm}{\end{displaymath}}
\newcommand {\ben}{\begin{equation}}
\newcommand {\een}{\end{equation}}

\newtheorem{corollary}{Corollary}[section]

\newfont{\smoldita}{cmmib8}
\newfont{\boldita}{cmmib10}
\newfont{\bboldita}{cmmib10 scaled\magstep1}
\newcommand{\sem}[1]{\mbox{$\{e^{t#1}\}_{t \geq 0}$}}

\newcommand{\mbb}[1]{\mathbb{#1}}
\newcommand{\mb}[1]{\mathbf{#1}}

\newcommand{\nn}{\nonumber}

\newcommand {\seq}[2]{\mbox{$({#1}_{#2})_{{#2}\in {\Bbb N}}$}}
\newcommand{\p}{\partial}

\newcommand{\cl}[2]{\int\limits_{#1}^{#2}}

\newcommand{\la}{\lambda}

\newcommand{\mc}[1]{\mathcal{#1}}

\begin{document}

\title[Diffusion and transport problems on networks]
 {Semigroup approach to diffusion and transport problems on networks}

%----------Author 1
\author[Jacek Banasiak]{Jacek Banasiak}

\address{
University of Kwazulu-Natal,  Durban, South Africa \\ and\\
 Technical University of \L\'{o}d\'{z}, \L\'{o}d\'{z}, Poland}

\email{banasiak@ukzn.ac.za}

\thanks{Research of J.B. and A.F was done during NRF/IIASA SA YSSP at the University of Free State and was partly supported by  National Science Centre of Poland through the grant N N201605640. Research of P.N. was supported by TWOWS and the UKZN Research Fund.
}
%----------Author 2
\author{Aleksandra Falkiewicz}
\address{
Technical University of \L\'{o}d\'{z}, \L\'{o}d\'{z}, Poland}

\email{aleksandra.falkiewicz@gmail.com}

\author{Proscovia Namayanja}
\address{University of KwaZulu-Natal, Durban, South Africa}
\email{proscovia@aims.ac.za}



\date{}

\maketitle

\begin{abstract}
Models describing transport and diffusion processes occurring along the edges of a graph and interlinked by its vertices have been recently receiving a considerable attention. In this paper we generalize such models and consider a network of transport or diffusion operators  defined on one dimensional domains and connected through boundary conditions linking the end-points of these domains in an arbitrary way (not necessarily as the edges of a graph are connected). We prove the existence of $C_0$-semigroups solving such problems and provide conditions fully characterizing when they are positive.

\end{abstract}

\section{Introduction}	
\label{sec1}
Recently there has been an interest in dynamical problems on graphs, where some evolution operators, such as transport or diffusion,  act on the edges of a graph and interact through its nodes. One can mention here quantum graphs, see e.g. \cite{Ka,Kos,Ku1,Ku2}, diffusion on graphs in probabilistic context, \cite{AB12, AG13, Kos},  transport problems, both linear and nonlinear, \cite{BaP, Bres, BD08, NagPhysD, KS04, MS07}, migrations, \cite{Ko}, and several other applications discussed in e.g. \cite{Ku2, DM}. In particular, the recent monograph \cite{DM} is a rich source of network models and methods. However, most of these works focus on a particular type of problems. For instance, in the quantum graph theory the main interest is to determine whether the operators defined on the edges of a graph are self-adjoint and the work is confined to the Hilbert space setting. Most papers on the linear transport theory on graphs focus on long term dynamics of the flow. Papers such as \cite{AB12, AG13, Kos}, motivated by probabilistic applications, look at Feller or Markov processes on graphs.

The present paper, which provides the theoretical foundation for \cite{BFN1}, is similar in spirit to \cite{AB12, AG13} in the sense that we prove the existence of strongly continuous semigroups in the space of continuous functions, as well as in the space of integrable functions, that solve the diffusion problem on a network. However, we extend the existing  results of \cite{AB12, AG13} by considering processes that are more general than the diffusion along edges of a geometric graph with Robin boundary conditions at its vertices in the sense that we allow for a communication between domains that are not necessarily physically connected. In fact, the models we analyse can be also interpreted as diffusion on a hypergraph, \cite{Be}, but we shall not pursue this line of research in this paper. For completeness, we also present similar results for transport problems, generalizing \cite{BaP} in a similar way.

To explain the idea of our extension, in the next section we consider two examples, see \cite{BF1, AB12}. 
\subsection{Motivation}
First, let us introduce basic notation which will help to formulate the problems and results.  We will work in a finite dimensional space, say, $\mbb R^m.$ The boldface characters will typically denote vectors in $\mbb R^m$, e.g. $\mb u = (u_1,\ldots,u_m).$ We denote  $\mc M = \{1,\ldots,m\}.$ Further, for any Banach space $X$, we will use the notation $\mb X = \underbrace{X\times\ldots\times X}_{m\,times}$, e.g. for $X=L_1(I), I=[0,1]$ we denote $\mb L_1(I)=\underbrace{L_1(I)\times\ldots\times L_1(I)}_{m\,times}$.

Let $(\mb A, D(\mb))$ be an operator in $\mb X$. If $\mb A$ is a generator, we will denote by \sem{\mb A} the semigroup generated by $\mb A$.

   \subsubsection {Diffusion}
                           We consider a finite graph without loops and isolated edges $\mc G = (\mc V, \mc E)$ with, say, $n$ vertices and $m$ edges.  On each edge there is a substance with density $u_j, j\in \mc M,$ which diffuses along this edge according to
                            \begin{equation}
 \p_t u_j = \sigma_j \p_{xx}u_j, \quad x\in ]0,1[, \quad j\in \mc M,
 \label{dif1}
 \end{equation}
 where $\sigma_j>0$ are constant diffusion coefficients, and can also enter the adjacent edges. To simplify considerations, each edge is identified with the unit interval $I.$
 In the model of \cite{AB12}, the particles can permeate between the edges across the vertices that join them according to a version of the Fick law.   To write down its analytical form, first we note that, since diffusion does not have a preferred direction, we can assign the tail, or the left endpoint, (that is, 0)  and  the head, or the right endpoint (that is, 1)  to the endpoints of the edge in an arbitrary way. Let $l_i$ and $r_i$ be the rates at which the substance leaves $e_i$ through, respectively, the left and the right endpoints and  $l_{ik}$ and $r_{ik}$ be the rates of at which it subsequently enters the edge $e_k$.
%\ben\label{trans1}
%\sum_{i\neq j}{l_{ij}} \leq l_i, \quad \sum_{i\neq j}{r_{ij}} \leq r_i.
%\een
Then the Fick law at, respectively,  the head and the tail of $e_i,$  gives
\begin{eqnarray}
-\p_x u_{i}(1) &=& r_{i}u_{i}(1) - \sum_{j\neq i}{r_{ij}u_{j}(v)},\nn
\\
\p_x u_{i}(0) &=& l_iu_i(0) - \sum_{j\neq i}{l_{ij}u_{j}(v)},
\label{trans3}
\end{eqnarray}
 where we have written $u_j(v)$ as $v$ may be either the tail or the head of the incident edge $e_j$.
  In particular, if there are no edges incident to the tail of $e_i$, or there are no edges incident to the head of $e_i$, then the Fick's laws take the form
\begin{equation}
\p_xu_{i}(0) = l_iu_i(0), \qquad -\p_xu_{i}(1) = r_{i}u_{i}(1),
\label{source}
\end{equation}
respectively, where either coefficient on the right hand side can be 0.

It is clear that if $r_{ij}\neq 0,$ then $l_{ij} =0$ and if $l_{ij}\neq 0$ then $r_{ij} =0$ and thus we can define an $m\times m$ matrix $\mbb A$ by setting $a_{ij}=1$ if either $r_{ij}=1$ or $l_{ij} =1$ and zero otherwise. Then $\mbb A$ is the adjacency matrix of the line graph $L(\mc G)$ of $\mc G$, see e.g. \cite{Bein}. However, it turns out that such a matrix is not easy to use and we see that introducing, for any $i,j \in \mc M,$
\begin{eqnarray}
k^{00}_{ij} &=& -l_{ij} \quad \mathrm{if}\; v = 0,\qquad    k^{01}_{ij} = -l_{ij}\quad  \mathrm{if}\; v=1, \qquad k^{00}_{ii} = l_{i},\nn\\
     k^{10}_{ij} &=& r_{ij}\quad \mathrm{if}\; v = 0,\qquad k^{11}_{ij} = r_{ij} \quad \mathrm{if}\; v=1, \qquad k^{11}_{ii} = -r_{i},
      \label{ktor}
      \end{eqnarray}
      where $v$ is either the tail or the head of the edge under consideration, the problem can be written as
            \begin{eqnarray}
 \p_t \mb u(x,t) &=& \mbb D\p_{xx}\mb u(x,t), \qquad (x,t) \in ]0,1[\times \mbb R_+,\nn\\
 \p_x \mb u(0,t)&=& \mbb K^{00}\mb u(0,t) + \mbb K^{01}\mb u(1,t),\qquad t>0, \nn\\
 \p_x \mb u(1,t)&=&  \mbb K^{10} \mb u(0,t)+\mbb K^{11} \mb u(1,t), \qquad t>0,\nn\\
 \mb u(x,0) &=& \mathring{\mb u}(x), \qquad x\in ]0,1[,
 \label{system1}
\end{eqnarray}
where $\mb u = (u_1,\ldots,u_m)$,  $\mbb D =  \mathrm{diag}\{\sigma_{i}\}_{1\leq i \leq m}$ and $\mathring{\mb u}$ is the initial distribution.

 It turns out that there is no mathematical reason why the matrices $K^{\omega}$, $\omega \in \Omega = \{00, 01,10,11\}$, in (\ref{system1}) should be restricted to the matrices given by (\ref{ktor}) which, indeed, form a strictly smaller class, see \cite{BF1}. In this paper we study the  well-posedness of (\ref{system1}) for arbitrary matrices $K^{\omega}$ in spaces $\mb C(I)$ and $\mb L_1(I),$ extending and simplifying the results of \cite{AB12, AG13}. We also find necessary and sufficient conditions for the semigroup solving (\ref{system1}) to be positive.
\subsubsection{Transport problems}

We consider a digraph $G =(V(G), E(G)) = (\{v_1,\ldots, v_n\}, \{e_1,\ldots, e_m\})$ with $n$ vertices and $m$ edges. We suppose that none of the vertices is isolated. As before, each edge is normalized so as to be identified with $I$ with the head at  $1$ and the tail  at  $0$. Following \cite{BaP, BD08, NagPhysD, KS04, MS07}, we consider  a substance of density $u_j(x,t)$ on the edge $e_j$, moving with speed $c_j$ along this edge. The conservation of mass at each vertex  is expressed by the Kirchhoff law,
\begin{equation}
\sum\limits_{j =1}^{m} \phi^-_{ij}c_ju_j(0,t)= \sum\limits_{j =1}^{m} \phi^+_{ij}c_ju_j(1,t), \quad t>0, i \in 1,\ldots,n,
\label{Kirch1}
\end{equation}
where  $\Phi^{-}= (\phi_{ij}^{-})_{1\leq i\leq n, 1\leq j\leq m}$ and $\Phi^{+}=(\phi_{ij}^{+})_{1\leq i\leq n, 1\leq j\leq m}$ are, respectively, the outgoing and incoming incidence matrices; that is, matrices with the entry $\phi_{ij}^-$ (resp. $\phi_{ij}^+$) equal to 1 if there is edge $e_j$ outgoing from (res. incoming to) the vertex $v_i$, and zero otherwise. Note that due to definitions of the matrices $\Phi^-$ and $\Phi^+,$ the summation on the right hand side is over all incoming edges of the vertex $v_i$ and on the left hand side over all outgoing edges of $v_i$.

In \cite{BaP} we considered a slightly more general model
\begin{eqnarray}
&&\p_t u_{j}(x,t)+ c_j\p_xu_{j}(x,t)=0,\quad x\in(0,1),\quad t\geq 0,\nn\\
&&u_{j}(x,0) = \mathring{u}_{j}(x),\nn\\
&&\phi_{ij}^{-}\xi_jc_ju_j(0,t) = w_{ij} \sum\limits_{k=1}^{m}{\phi_{ik}^{+}(\gamma_{k}c_ku_k(1,t))},
\label{ff2}
\end{eqnarray}
where $\gamma_j>0$ and $\xi_j>0$ are the absorption/amplification coefficients at, respectively, the head and the tail of  $e_j.$ Here the matrix $\{w_{ij}\}_{1\leq i\leq n, 1\leq j\leq m}$  describes the  distribution of the incoming flow at the vertex $v_i$ into the edges outgoing from it; it is a column stochastic matrix, \cite{KS04}. We denote $\mbb C = \mathrm{diag} \{c_j\}_{1\leq j\leq m},  \Xi = \mathrm{diag} \{\xi_j\}_{1\leq j\leq m}$ and $\Gamma  = \mathrm{diag} \{\gamma_j\}_{1\leq j\leq m}.$

It follows, \cite{BaP}, that if $G$ has a sink, than there is no $C_0$-semigroup solving (\ref{ff2}). Hence we discard this case and then it can be proved, e.g. \cite[Proposition 3.1]{BD08}, that (\ref{ff2}) can be written as an abstract Cauchy problem
\begin{equation}
\mb u_t = \mb A\mb u, \qquad \mb u(0) = \mathring{\mb u},
\label{ACP1a}
\end{equation}
in $\mb X = \mb L_1(I)$, where $\mb A$ is the realization of  ${\sf A} = \mathrm{diag}\{-c_j \p_x\}_{1\leq j\leq m}$ on the domain
\begin{equation}
D(\mb A) = \{\mb u \in \mb W^{1}_1(I);\; \mb u(0) = \Xi^{-1}\mbb C^{-1} \mbb B\Gamma\mbb C\mb u(1)\},
\label{DA''}
\end{equation}
where $\mbb B$ is the (transposed) adjacency matrix of the line graph of $G$.

As with the diffusion problems, there is no mathematical reason to restrict our analysis to the matrices of the form $\Xi^{-1}\mbb C^{-1} \mbb B\Gamma\mbb C$ in the boundary conditions which, in fact, see \cite{BF1}, form a strict subset of the set of all matrices. Thus, we consider the following  generalization of (\ref{ACP1a}), (\ref{DA''}),
\begin{equation}
\mb u_t = \mb A\mb u, \qquad \mb u(0)=\mbb K \mb u(1),\qquad \mb u(0) = \mathring{\mb u},
\label{ACP1}
\end{equation}
 where $\mbb K$ is an arbitrary matrix.

\begin{example}
The main difference between (\ref{ACP1}) with arbitrary $\mbb K$ and the model with $\mbb K$ given in (\ref{DA''}) is that in the former, the exchange of the substance can occur instantaneously between any edges, while in the latter the edges must be physically connected by vertices for the exchange to take place. So, for instance, if there is a connection $e_1\to e_2$ and $e_2\to e_3$, then there is no connection $e_1\to e_3$.  So, while in general (\ref{ACP1}) cannot model a flow in a physical network, it can describe e.g. a mutation process. Indeed, let  a population of cells be divided into $m$ subpopulations, with $\mb v(x) =(v_1(x),\ldots, v_m(x)),$ where $v_j(x),  j\in \mc M, x\in [0,1],$ is the density of cells of age $x$ whose genotype belongs to a  class $j$ (for instance, having $j$ copies of a gene of a particular type). We assume that cells of class $j$ mature and divide upon reaching maturity at $x=1,$ with offspring, due to mutations,  appearing in class $i$ with probability $k_{ij}$, $i\in \mc M$. In such a case $(k_{ij})_{1\leq i,j\leq m}$ is a column stochastic matrix. We note that a particular case of this model is the discrete Rotenberg-Rubinov-Lebowitz model, \cite{rot}, where the cells are divided into classes according to their maturation velocity.

A similar interpretation can be given to (\ref{system1}) where the variable $x$, instead of the age, denotes the size of the organism, e.g. \cite{Had}.

Moreover, problems of the form (\ref{ACP1}) arise e.g. in  queuing theory, \cite{da}.

\end{example}

\section{Well-posedness of the diffusion problem}

We shall consider solvability of (\ref{system1})  in $\mb X=\mb C(I)$ and $\mb X= \mb L_1(I).$ For technical reasons, we also shall need the solvability of (\ref{system1}) in $\mb W^1_1(I)$.  The norms in these spaces will be denoted, respectively, by $\|\cdot\|_\infty, \|\cdot\|_0, \|\cdot\|_1$. However, if it does not lead to any misunderstanding, we will use $\mb X$ to denote any of these spaces and $\|\cdot\|$ or $\|\cdot\|_\mb X$ to denote the norm in $\mb X$. Similarly, by $|||\cdot|||_{\mb X}$ we denote the operator norm in the space of bounded linear operators from $\mb X$ to $\mb X$.

Our results are based on \cite{Gre} and thus, introducing relevant spaces and operators, we try to keep notation consistent with \textit{op.cit.} First, consider
\begin{equation}
\mb X \ni \mb u \to L\mb u = (\gamma_0 \p_x\mb u, \gamma_1 \p_x \mb u) \in \mb Y = \mbb R^m\times \mbb R^m,
\label{Phi}
\end{equation}
where $\gamma_i, i=0,1,$ is the trace operator at $x=i$ (taking the value at $x=i$ if $\mbb X = \mb C(I)$). The domains of $L$ are $D(L) = \mb C^1(I)$ if $\mb X =\mb C(I)$ and $D(L) = \mb W^2_1(I)$ in two other cases. Then we define the operator
$$
\mb X\ni \mb u \to \Phi\mb u = \mbb K(\gamma_0 \mb u, \gamma_1\mb u) = \left(\begin{array}{cc}\mbb K^{00}& \mbb K^{01}\\\mbb K^{10}& \mbb K^{11}\end{array}\right)\left(\begin{array}{c}\gamma_0 \mb u\\\gamma_1 \mb u\end{array}\right) \in \mbb R^m\times \mbb R^m.
$$
Further, let us denote
\begin{equation}
\Phi^*\mb u = \mbb K^*(\gamma_0 \mb u, \gamma_1\mb u) = \left(\begin{array}{cc}\mbb K^{00\,T}& -\mbb K^{10\,T}\\-\mbb K^{01\,T}& \mbb K^{11\,T}\end{array}\right)\left(\begin{array}{c}\gamma_0 \mb u\\\gamma_1 \mb u\end{array}\right),
\label{phistar}
\end{equation}
where $\mbb K^T$ denotes the transpose of $\mbb K$. Clearly $(\Phi^*)^* = \Phi$.

Let $\mathsf A$ denote the differential expression $\mathsf A\mb u := \mbb D\p_{xx}\mb u.$ Then we define the operators $\mb A^\alpha_\Phi$, $\alpha =\infty, 0, 1$ by the restriction of $\mathsf A$ to the domains
\begin{eqnarray}
D(\mb A^\infty_\Phi) &=& \{\mb u \in \mb C^2(I);\;L\mb u = \Phi \mb u\},\nn\\
D(\mb A^0_\Phi) &=& \{\mb u \in \mb W^2_1(I);\;L\mb u = \Phi \mb u\},\nn\\
D(\mb A^1_\Phi) &=& \{\mb u \in \mb W^3_1(I);\;L\mb u = \Phi \mb u\},\label{real}
\end{eqnarray}
respectively. As before, we drop the indices from the notation if it will not lead to any misunderstanding.

\subsection{Basic estimates in the scalar case}
Consider the general resolvent equation for (\ref{system1})
\begin{equation}
\la \mb u - \mbb D\p_{xx} \mb u = \mb f, \quad x \in ]0,1[.
\label{reseq1}
\end{equation}
Since without the boundary conditions the system is uncoupled, its solution $\mb u$ is given by
\begin{equation}
\mb u(x) = \mbb E(-\bmu x) \mb C_1+ \mbb E(\bmu x)\mb C_2 + \mb U_{\bmi}(x)
\label{gensol1}
\end{equation}
where $\mbb E(\pm\bmu x) = \mathrm{diag}\{e^{\pm \mu_ix}\}_{1\leq i\leq m},$ $0\neq \la/\sigma_i = \mu_i^2 = |\la/\sigma_i|e^{i\theta}$ with $\Re \mu_i >0,$ $\mb U_{\bmi}(x) = (U_{\mu_1}(x),\ldots, U_{\mu_m}(x))$ with
\begin{equation}
 U_{\mu_i}(x) = \frac{1}{2\mu_i\sigma_i} \cl{0}{1} e^{-\mu_i|x-s|}f_i(s)ds,
\label{Umu}
\end{equation}
and the vector constants $\mb C_1$ and $\mb C_2$ are determined by the boundary conditions. Further, denote
$$
\Sigma_{\alpha,\theta_0}= \{\la = |\la|e^{i\theta}\in \mbb C;\;|\la|\geq \alpha, |\theta|\leq \theta_0 <\pi\}$$
and $ \Sigma_{0, \theta_0} = \Sigma_{\theta_0}.$

The starting point are well-known estimates for the scalar Dirichlet and Neumann problems, e.g. \cite{EN}. We briefly recall them here in a slightly more precise form, similarly to \cite{BobAPM}.  In the scalar case we can use $\sigma =1$ as will become clear when the estimate is derived.  It follows that for the Dirichlet problem we have
\begin{equation}
C_1^D = \frac{U_\mu(1)-U_\mu(0)e^{\mu}}{e^\mu -e^{-\mu}}\qquad C_2^D = \frac{U_\mu(0)e^{-\mu}-U_\mu(1)}{e^\mu -e^{-\mu}}
\label{CD}
\end{equation}
and for the Neumann problem
$$
C_1^N = \frac{U_\mu(0)e^\mu +U_\mu(1)}{e^\mu -e^{-\mu}}\qquad C_2^N = \frac{U_\mu(1)+U_\mu(0)e^{-\mu}}{e^\mu -e^{-\mu}}.
$$
Further, \cite{EN},
\begin{equation}
\|U_\mu\|_\infty \leq \frac{\|f\|_\infty}{|\mu|\Re \mu} \leq \frac{\|f\|_\infty}{|\la|\cos \theta/2}, \qquad
\|U_\mu\|_0 \leq \frac{\|f\|_0}{|\mu|\Re\mu} \leq \frac{\|f\|_0}{|\la|\cos \theta/2}
\label{Uest0}
\end{equation}
and, for $ i=0,1$,
\begin{equation}
|U_\mu(i)| \leq \frac{\|f\|_\infty(1-e^{-\Re\mu})}{2|\mu|\Re \mu},\quad |U_\mu(i)| \leq \frac{\|f\|_0}{2|\mu|}.
\label{U}
\end{equation}
The following result plays an essential role in deriving precise estimates,
\begin{equation}
\left |\frac{1}{e^\mu-e^{-\mu}}\right| = e^{-\Re \mu}\left|\sum\limits_{n=0}^{\infty}e^{-2n\mu}\right| \leq
e^{-\Re \mu}\sum\limits_{n=0}^{\infty}e^{-2n\Re \mu} = \frac{e^{-\Re \mu}}{1-e^{-2\Re\mu}}.
\label{estemu}
\end{equation}
Then, for either Dirichlet or Neumann problem in $C(I)$, for $\lambda\in \Sigma_{\theta_0}$ for any $\theta_0<\pi$, we have for $\omega = D,N$
\begin{eqnarray}
\|u^\omega\|_\infty &\leq & |C^\omega_1| +|C^\omega_2| e^{\Re \mu} + \frac{\|f\|_\infty}{|\la|\cos \theta_0/2}\label{uomegainf}\\
& \leq& \frac{\|f\|_\infty}{|\la|\cos \theta_0/2}\left( \frac{1}{1-e^{-2\Re \mu}} \left(\frac{1-e^{-2\Re\mu}}{2} + \frac{1-e^{-2\Re\mu}}{2}\right) + 1\right)\nn\\ &\leq&  \frac{2\|f\|_0}{|\la|\cos \theta_0/2}.\nn
\end{eqnarray}
Similarly, in $L_1(I)$, we have
\begin{eqnarray}
\|u^\omega\|_0 &\leq & |C^\omega_1|\frac{1-e^{-\Re\mu}}{\Re\mu} +|C^\omega_2| \frac{e^{\Re \mu}-1}{\Re\mu} + \frac{\|f\|_0}{|\la|\cos \theta_0/2}\label{uomegajed}\\
& \leq& \frac{\|f\|_0}{|\la|\cos \theta_0/2}\left( \frac{1}{1-e^{-2\Re \mu}} \left(\frac{1-e^{-2\Re\mu}}{2} + \frac{1-e^{-2\Re\mu}}{2}\right) + 1\right)\nn\\
 &\leq&  \frac{2\|f\|_0}{|\la|\cos \theta_0/2}.\nn
\end{eqnarray}
    Consider now $A^1_0,$ see (\ref{real}); that is, the operator corresponding to the Neumann boundary conditions in one dimension. We have
\begin{prop}
$A^1_0$ generates an analytic semigroup in $W^1_1(I)$ with the resolvent satisfying the estimate
\begin{equation}
\|R(\la, A^1_0)f\|_{1} \leq \frac{2\|f\|_{1}}{|\la|\cos \theta_0/2}, \qquad \la \in \Sigma_{\theta_0}.
\label{resw11}
\end{equation}
\end{prop}
\begin{proof}
Consider the resolvent equation
$$\la u -  \p_{xx}u = f, \quad \p_x u(0)=\p_x u(1) = 0.
$$
If $f \in W^1_1(I)$, then $u\in W^3_1(I)$ and we can differentiate the differential equation getting, for $v:=\p_xu$,
\begin{equation}
\la v - \p_{xx} v = \p_x f, \quad v(0)=v(1) =0;
\label{dirv}
\end{equation}
that is, $v$ satisfies the resolvent equation for the Dirichlet problem. Hence
$$
\|u\|_{1} = \|u\|_0 + \|\p_xu\|_0 \leq \frac{2\|f\|_0}{|\la|\cos \theta_0/2} + \frac{2\|\p_xf\|_0}{|\la|\cos \theta_0/2} = \frac{2\|f\|_{1}}{|\la|\cos \theta_0/2}.
$$
\end{proof}
Since $\la u -  \sigma\p_{xx}u = f$ is equivalent to $\sigma^{-1}\la u -  \p_{xx}u = \sigma^{-1}f$, we see that the estimates above are independent of $\sigma$.

\subsection{Solvability of (\ref{system1})}

The ideas in this section are based on \cite{AB12, AG13} but the analysis is simplified
by using the analyticity of the semigroup and the application of \cite[Theorem 2.4]{Gre}.

Since in the vector case and Neumann boundary conditions the resolvent equation decouples, we obtain
   \begin{equation}
\|R(\la, \mb A_0)f\|_{\mb X} \leq \frac{2\|f\|_{\mb X}}{|\la|\cos \theta_0/2}, \qquad \la \in \Sigma_{\theta_0}
\label{resw11}
\end{equation}
for any $\theta_0<\pi$ and $\mb X=\mb C(I), \mb L_1(I), \mb W^1_1(I)$. Hence, in particular $\mb A_0$ generates an analytic semigroup in $\mb X$.

 We begin with a straightforward consequence of Ref.~\cite{Gre}.
\begin{thm}
Let $\mb X$ be either $\mb C(I),$ or $\mb W^1_1(I)$. Then the operator $\mb A_\Phi$ generates an analytic semigroup in $\mb X$ with the resolvent satisfying the estimate
\begin{equation}
|||R(\la, \mb A_\Phi)|||_{\mb X} \leq  2|||R(\la, \mb A_0)|||_{\mb X}, \qquad \la \in  \Sigma_{\alpha, \theta_0}
\label{grei0}
\end{equation}
for some $\alpha\geq 0$.\label{th42}
\end{thm}
\begin{proof}
The result follows directly from \cite[Theorem 2.4]{Gre}, since $\mb A_0$ generates an analytic semigroup in $\mb X$, $L$ is an unbounded operator on $\mb X$ which is, however, bounded as an operator from $D(\mb A)$ to $\mb Y$, and it is a surjection (the right inverse in each case is given by $[L^{-1}_r(\mb y_0,\mb y_1)](x) = \frac{1}{2}(\mb y_1- \mb y_0)x^2 -\mb y_0 x.$  Furthermore,  $\Phi$ is a bounded operator on $\mb X$ (if $\mb X = \mb W^1_1(I)$  this follows since $\mb W^1_1(I)$ functions are absolutely continuous on $I$). Hence, the assumptions  \cite[(1.13)]{Gre} are satisfied and the theorem follows from \cite[Theorem 2.4]{Gre}.
\end{proof}
In $\mb X = \mb L_1(I)$ the situation is more complicated as $\Phi$ is not bounded on $\mb X$. First we show that $R(\la, \mb A^1_\Phi)$ extends to a resolvent on $\mb L_1(I)$.
\begin{lem} We have
$$
\overline{\mb A^1_\Phi}^{\mb L_1(I)} = \mb A^0_\Phi
$$
and the resolvent set of $\mb A^0_\Phi$ satisfies $\rho (\mb A^0_\Phi)\cap \mbb R_+ \neq \emptyset$.
\label{lemres1}
\end{lem}
\begin{proof}
Let $D(\mb A^1_\Phi)\ni \mb u_n \to \mb u \in \mb L_1(I)$ and $\mb A_\Phi^0 \mb u_n = \mb A \mb u_n \to \mb v \in \mb L_1(I)$. This shows that
$\mb u_n\to u$ in $\mb W^2_1(I)$ and thus $\mb v = \mb A\mb u$. Since taking the trace of a $\mb W^2_1(I)$ function and of its derivative is continuous in $\mb W^2_1(I)$, the boundary values of $\mb u_n$ and $\p_x\mb u_n$ are preserved in the limit and thus $\mb u \in D(\mb A^0_\Phi)$. Hence
$$
\overline{\mb A^1_\Phi}^{\mb L_1(I)} \subset \mb A^0_\Phi.
$$
On the other hand, let $\mb u \in D(\mb A^0_\Phi)$. Then $\mb v =\mb u -\mb f \in \stackrel{0}{W}\phantom{W}\!\!\!\!\!\!\!\!^2_1(I)$, where
\begin{eqnarray*}
\mb f(x) &=& (\mb u(0) + \mb u'(0) -2\mb u(1) +\mb u'(1))x^3\\
 &&-(2\mb u(0) + 2\mb u'(0) -3\mb u(1) +\mb u'(1))x^2 + \mb u'(0)x + \mb u(0) \in D(\mb A^0_\Phi).
\end{eqnarray*}
Thus there is a sequence $\seq{\mb h}{n} \subset \mb C^\infty_0(I)$ converging to $\mb v$  in $\mb W^2_1(I)$ and hence $\mb u$ is the limit in $\mb W^2_1(I)$ of functions $\mb u_n = \mb h_n + \mb f \in D(\mb A^1_\Phi)$. Since the convergence in $\mb W^2_1(I)$ implies the convergence of both $\mb u_n$ and $\mb A\mb u_n$ in $\mb L_1(I)$ we see that also
$$
\overline{\mb A^1_\Phi}^{\mb L_1(I)} \supset \mb A^0_\Phi.
$$
Finally, we see that, as in the scalar case (\ref{res1}), the solvability of the resolvent problem for (\ref{system1})
is equivalent to solvability of the linear system
\begin{eqnarray}
-\mbb M \mb C_1 + \mbb M \mb C_2 &=& \mbb K^{00}(\mb C_1 + \mb C_2) +\mbb  K^{01}(e^{-\mu} \mb C_1 + e^\mu \mb C_2)\nn\\
 &&-\mbb M \mb U_{\bmi} (0) + \mbb K^{00} \mb U_{\bmi}(0)+\mbb  K^{01} \mb U_{\bmi}(1),\nn\\
-\mbb M \mbb E(-\bmu)\mb C_1 + \mbb M \mbb E(\bmu) \mb C_2 &=& \mbb K^{10}(\mb C_1 + \mb C_2) + \mbb K^{11}(\mbb E(-\bmu) \mb C_1 + \mbb E(\bmu) \mb C_2)\nn\\&& +\mbb M \mb U_{\bmi} (1) +\mbb  K^{01} \mb U_{\bmi}(0)+\mbb  K^{11} \mb U_{\bmi}(1),\label{resfor}
\end{eqnarray}
for the vectors $\mb C_1, \mb C_2$, where $\mbb M=\mathrm{diag}\{\bmu\}$. Once the constants are found, the solution is given by the vector version of (\ref{gensol1}) and thus belongs to $\mb W^2_1(I)\subset \mb L_1(I)$. The system above is solvable provided its determinant is different from zero. However, the determinant is clearly an entire function in $\mu$ and thus can have only isolated zeros in $\mbb C$. Hence,  there must be positive values of $\la \in \rho (\mb A^0_\Phi)$.
\end{proof}

\begin{thm}
The operator $\mb A^0_\Phi$ generates an analytic semigroup in $\mb L_1(I).$ \label{thm26}
\end{thm}
\begin{proof}
We use the formula from Lemma 1.4 of Ref.~\cite{Gre} stating that
\begin{equation}
R(\la, \mb A^1_\Phi) = (I-L_\la\Phi)^{-1}R(\la, \mb A^1_N)
\label{grei1}
\end{equation}
in $\mb X = \mb W^1_1(I)$ as $(I-L_\la\Phi)^{-1}$ is an isomorphism of $\mb X$ for  large $|\la|$, where $L_\la = (L|_{Ker (\la - \mb A)})^{-1}, \la \in \rho(\mb A^1_N)$, see the proof of \cite[Theorem 2.4]{Gre}.

However, $R(\la, \mb A^1_N)$ extends by density to $R(\la, \mb A^0_N)$ (the resolvent on $\mb L_1(I)$) which is a bounded linear operator from $\mb L_1(I)$ to $D(\mb A^1_N)\subset \mb W^2_1(I)$ which is continuously embedded in $\mb W^1_1(I)$ and thus $R(\la, \mb A^1_\Phi)$ extends to a bounded linear operator, say $R_\Phi(\la)$, on $\mb L_1(I)$. Since $R(\la, \mb A^1_\Phi)$ is a resolvent on $\mb W^1_1(I)$, we obtain, by density, that $R_\Phi(\la)$ is a pseudoresolvent on $\mb L_1(I)$. Furthermore, $\mb C_0^\infty(]0,1[)\subset D(\mb A^1_\Phi),$ thus it is also a subset of the range of $R_\Phi(\la)$ and therefore the range is dense in $\mb L_1(I)$. Also, since the range of $R(\la, \mb A^0_N)$ is in $\mb W^2_1(I)\subset \mb W^1_1(I),$ there is no need to extend $(I-L_\la\Phi)^{-1}$. Hence $R_\Phi(\la)$ is a one-to-one operator as a composition of two injective operators. Thus, by Proposition III.4.6 in Ref.~\cite{EN}, $R_\Phi(\la)$ is the resolvent of a densely defined operator, say $\hat{\mb A}_\Phi$. Clearly, $\hat{\mb A}_\Phi$ is a closed extension of $\mb A^1_\Phi$ and thus $\hat{\mb A}_\Phi\supset {\mb A}^0_\Phi = \overline{\mb A^1_\Phi}^{\mb L_1(I)}. $ From Lemma \ref{lemres1}, there is $\la \in \rho(\hat{\mb A}_\Phi) \cap \rho(\mb A^0_\Phi),$ thus $R(\la, \hat{\mb A}_\Phi)\supset   R(\la, {\mb A}^0_\Phi)$ and hence $\hat{\mb A}_\Phi =   {\mb A}^0_\Phi$ and, consequently, $R(\la, \mb A^0_\Phi)$ is defined in some sector. To prove that  is a sectorial operator, we use the idea of \cite{AG13} but in a somewhat simpler way. We consider the operator $\mb A^\infty_{\Phi^*}.$ Let $R^{\#}_\la$ denote the adjoint to $R(\la, \mb A^\infty_{\Phi^*})$ for $\la \in \rho(\mb A^\infty_{\Phi^*})$, which acts in the space of signed (vector) Borel measures on $I$, \cite{Bobks}. $R(\la, \mb A^\infty_{\Phi^*})$ is given by (\ref{gensol1}) with $\mb C_1, \mb C_2$ given by (\ref{resfor}) with the matrices $\mbb K^\omega$ replaced by corresponding matrices in (\ref{phistar}). Hence, apart from $\mb U_{\bmi}(x)$,  $R(\la, \mb A^\infty_{\Phi^*})$ is a composition of an algebraic operator coming from inverting the matrix in (\ref{resfor}) with a vector of functionals acting on $\mb f$. Thus, if  $\mb f\in \mb L_1(I)$ is the density of an absolutely continuous measure, a standard calculation shows that
$$R^{\#}_\la \mb f = R(\bar \la, \mb A^0_\Phi)\mb f$$
 From the definition of the norm of a signed Borel measure, if the latter is absolutely continuous, its norm is equal to the $L_1$ norm of its density. Since taking the adjoint preserves the norm of the operator, we obtain
\begin{equation}
|||R(\la, \mb A^0_\Phi)|||_{\mb L_1(I)} = |||R(\la, \mb A^\infty_{\Phi^*})|||_{\mb C(I)}.
\label{resnorm}
\end{equation}
Hence, by Theorem \ref{th42}, $\mb A^0_\Phi$ is sectorial and, being densely defined, it generates an analytic semigroup on $\mb L_1(I)$.
\end{proof}

\subsection{Positivity of the semigroup}
Let us recall that for an element $\mb u$ of a Banach lattice, we write $\mb u >0$ if $0\neq \mb u\geq 0.$ We have the following result
\begin{thm}
The semigroup $\sem{\mb A^\infty_\Phi}$ is resolvent positive if and only if
 \begin{eqnarray}-\mbb K^{00},  &&\mbb K^{11}\;\mathrm{are\;nonnegative\; off\!-\!diagonal}\nn\\
 -\mbb K^{01}, &&\mbb K^{10}\;\mathrm{ are\; nonnegative}.
\label{maxcon}
\end{eqnarray}
\label{positive}
\end{thm}
\begin{proof}
 To prove the result we use Theorem B-II.1.6 in \cite{N}, which states that if an operator $A$ on $C(K)$  (where $K$ is compact) generates a  semigroup, then the semigroup is positive (or, equivalently, $A$ is resolvent positive) if and only if $A$ satisfies positive minimum principle: for every $0\leq f\leq D(A)$ and $x\in K$, if $f(x)=0$, then $(Af)(x)\geq 0$. However, to be rephrase the problem at hand  in the language of the positive maximum principle,  $C(K)$ must to be the space of real valued  functions.
 To achieve this, here we identify $\mb C(I)$ with the space $C(\mb I),$ where $\mb I = [0_1,1_1]\cup\ldots\cup [0_m,1_m]$; that is, instead of considering a vector function on $I$ we consider a scalar function on a disconnected compact space composed of $m$ disjoint closed intervals. In particular, each edge $e_j, j\in \mc M$ is identified with the closed interval $[0_j,1_j].$ Then  $\mb A^\infty_\Phi$ will be changed to $A^\infty_\Phi$ which is  the restriction of $$
(Au)(x) = \sum\limits_{j\in \mc M}\chi_{[0_j,1_j]}(x)\sigma_j\p_{xx}u(x)$$
will be the space of twice differentiable functions on $Int \mb I = ]0_1,1_1[\cup\ldots\cup ]0_m,1_m[$, differentiable on $\mb I$ and satisfying
\begin{eqnarray*}
\p_xu(0_j) &=& \sum\limits_{k=1}^m k^{00}_{jk}u(0_k) + \sum\limits_{k=1}^m k^{01}_{jk}u(1_k),\\
\p_xu(1_j) &=& \sum\limits_{k=1}^m k^{10}_{jk}u(0_k) + \sum\limits_{k=1}^m k^{11}_{jk}u(1_k).
\end{eqnarray*}
  Thus, if $0\leq u \in D(A)$ takes value 0 at some $x\in \mb I$, then either $x \in Int \mb I$ and then classically $\p_{xx} u (x)\geq 0,$ or $x = 0_j$ or $x=1_j$ for some $j \in \mc M$. If $x=0_j$, then  $\p_{x} u(0_j) \geq 0$. If  $\p_{x} u(0_j) = 0,$ then $\p_{xx} u(0_j)\geq 0$ follows as in Example B-II.1.24 of Ref.~\cite{N} (or simply by noting that the even extension to $[-1,0]$ gives a $C^2$ function on $[-1,1]$ with minimum at $x=0$). If $\p_x u(0_j)<0$ then, since  $u(0_j)=0,$  such a function cannot be nonnegative on $[0_j,1_j]$. Analogous considerations hold if $ u(1_j)=0$. Therefore the positive minimum principle is satisfied and (\ref{maxcon}) yields the positivity of \sem{A^\infty_\Phi} and thus of \sem{\mb A^\infty_\Phi}.

  To prove the converse we introduce the following notation. For any $\bal :=(\bal^0,\bal^1)\geq 0$ with $\alpha^j_i=0$ for some $i \in \mc M, j=0,1,$ we have
\begin{equation}
\Xi:=\left(\begin{array}{cc}-\mbb K^{00}&-\mbb K^{01}\\\mbb K^{10}&\mbb K^{11}\end{array}\right).
\label{eqmaxcon}
\end{equation}
Accordingly, we denote
$$
(\Xi\bal)^r_s = (-1)^{r+1}\sum\limits_{l=0,1}\left(\sum\limits_{j\in \mc M}k^{rl}_{sj}\alpha^l_j \right), \quad r=0,1, s\in \mc M.
$$
Let us assume that (\ref{maxcon}) is not satisfied. Then there is a non-diagonal element of $\Xi$ which is strictly negative.
Suppose $(-1)^{r+1}k^{rs}_{ij}<0$ for some $i\neq j$ and $r,s =0,1,$ and consider a vector $\bal$ with
\begin{equation}
\alpha^s_j=1, \alpha^r_i=0, \alpha^t_l = \delta>0, \quad t=0,1, \quad l\neq j\;\mathrm{if}\;t = s\;\mathrm{and}\; l\neq i\;\mathrm{if}\;t= s.
\label{alpha}
\end{equation}
 Then for $r=s$ we obtain
\begin{equation}
(\Xi\bal)^r_i= (-1)^{r+1}\left(k^{rr}_{ij} +\delta\left(\sum\limits_{l\in \mc M, l\neq j,i}k^{rr}_{il} +  \sum\limits_{l\in \mc M}k^{rt}_{il}\right)\right)<0
\label{xiri1}
\end{equation}
where $t=0$ if $r=1$ and $t=1$ if $r=0$, and for $r\neq s$
\begin{equation}
(\Xi\bal)^r_i= (-1)^{r+1}\left(k^{rt}_{ij} +\delta\left(\sum\limits_{l\in \mc M, l\neq i}k^{rr}_{il} +  \sum\limits_{l\in \mc M, l\neq j}k^{rt}_{il}\right)\right)<0
\label{xiri2}
\end{equation}
for sufficiently small $\delta>0$. We shall prove that there exists a function $0\leq u \in C^\infty(\mb I)$ satisfying $u (r_i) = \alpha^r_i$ and $(-1)^{r+1}\p_x u(r_i) = (\Xi\bal)^r_i$ which additionally satisfies $\p_{xx}u (r_i)<0$.

For a given constants $\alpha_i^r,\beta_i^r =   (\Xi\bal)^r_i, r=0,1,$ we consider auxiliary functions $f_i^r(x) = \beta_i^rx(x-1) +\alpha_i^r$. We have $f_i^r(r) = \alpha_i^r, \p_x f_i(r) = (-1)^{r+1}\beta_i^r$ and $\p_{xx} f_i^r = 2\beta_i^r$. We observe that as long as $\alpha_i^r>0$,  there is a one-sided interval ($(0,\omega_i)$ if $r=0$ and $(\omega_i, 1)$ if $r=1$), where $f_i^r\geq 0$ irrespective of the sign of $\beta_i^r$. On the other hand, if $\alpha_i^r=0$, for local nonnegativity we need $\beta_i^r\leq 0$. Now let $\phi$ be a nonnegative $C^\infty$ function which is $1$ on $[-a,a]$ and $0$ outside $[-2a,2a]$ where $0<2a <\min_{i\in \mc M}\omega_i$   and define
$$
 u (x) = \phi (x) f_j^0 (x) + \phi(1-x)  f^1_j(x), \quad x  \in [0_j,1_j], j \in \mc M.
$$
is a $C^\infty(\mb I)$ function that satisfies:\begin{enumerate}
\item $u\geq 0$;
\item $ u(0_j) = \alpha^0_j$ and $\p_x  u(0_j) = \p_x f^0_j(0_j) = -\beta^0_j =  (\Xi\bal)^0_j$;
\item $ u(1_j) =\alpha^1_j$ and $\p_x u(1_j) = \p_x f^1_j(1_j) = \beta^1_j = (\Xi\bal)^1_j,$
\end{enumerate}
so that $0\leq  u \in D(A).$ Recalling that we assumed $(-1)^{r+1}k^{rs}_{ij}<0$ for some $i\neq j$ and $r,s =0,1,$ we consider coefficients $\{\alpha^t_j\}_{t=0,1, j\in \mc M}$ satisfying (\ref{alpha}). Then we have $ u(r_i) =0$. On the other hand, $\p_{xx} u(r_i) = 2\beta^s_i = 2(\Xi\bal)^r_i<0$ by  (\ref{xiri1}) or (\ref{xiri2}). Thus, there is a nonnegative element  $u\in D(A)$ for which $\p_{xx} u <0$ at a point where the global minimum of zero is attained.
\end{proof}
We can use this result to prove an analogous result in $\mb L_1(I)$.
\begin{cor}
The operator $A^0_\Phi$ generates a positive semigroup if and only if the assumptions of Theorem \ref{positive} are satisfied.
\end{cor}
\begin{proof}
In one direction the result immediately follows by density of $\mb C(I)$ in $\mb L^1(I)$. Conversely, if $\sem{A^0_\Phi}\geq 0$ then, in particular, for any $0\leq \,\mathring{\mb u}\,\in\! \mb C(I)$ we have  $e^{tA^0_\Phi}\mathring{\mb u} = e^{tA^\infty_\Phi}\mathring{\mb u}\,\geq 0$.
 \end{proof}

\section{Solvability of (\ref{ACP1a})}

We return to the problem (\ref{ACP1a}) where, we emphasize, $\mbb K$ is an arbitrary matrix. In this section we restrict our attention to $\mb X = \mb L_1(I)$ as in $\mb C(I)$ the operator $\mb A$ is not densely defined. Let us recall that $\mb A$ is the realization of  ${\sf A} = \mathrm{diag}\{-c_j \p_x\}_{1\leq j\leq m}$ on the domain
$D(\mb A) = \{\mb u \in \mb W^{1}_1(I);\; \mb u(0) = \mbb K\mb u(1)\}$.

The following theorem for (\ref{ff2}) has been proved in \cite{BaP} (see also \cite{BaLNM}) but the proof for any nonnegative matrix $\mbb K$ is practically the same. Here we extend this proof to an arbitrary $\mbb K$. However, for the proof in the general case we need to provide basic steps of the proof for nonnegative matrices.

\begin{theorem} The operator $(\mb A, D(\mb A))$ generates a $C_0$-semigroup on $\mb L_1(I)$. The semigroup is positive of and only if $\mbb K\geq 0$.\label{thtragen}
\end{theorem}\begin{proof}
Clearly, $\mb C_0^\infty(]0,1[) \subset D(\mb A)$ and hence $D(\mb A)$ is dense in $\mb X$. To find the resolvent of $\mb A,$ the first step is to solve
\begin{equation}
\la u_j + c_j \p_x u_j = f_j, \quad j=1,\ldots,m,\quad x \in ]0,1[,
\label{res0}
\end{equation}
with $\mb u \in D(A)$.
 Integrating, we find
\begin{equation}
\mb u(x) = \mbb E_\la(x)\mb v +\mbb C^{-1} \cl{0}{x}\mbb E_\la(x-s)\mb f(s)ds,
\label{res1}
\end{equation}
where $\mb v = (v_1,\ldots,v_m)$ is an arbitrary vector and $\mbb E_{\la}(s) = \mathrm{diag}\left\{e^{-\frac{\la}{c_j}s}\right\}_{1\leq j\leq m}.$  To determine  $\mb v$ so that $\mb u \in D(\mb A)$, we use the boundary conditions. At $x=1$ and at $x=0$ we obtain, respectively $$
\mb u(1) = \mbb E_\la(1)\mb v + \mbb C^{-1}\cl{0}{1}\mbb E_\la(1-s)\mb f(s)ds, \qquad \mb u(0) = \mb v.
$$
Then, using the boundary condition $\mb u(0) = \mbb K\mb u(1)$ we obtain
\begin{equation}
(\mbb I - \mbb K\mbb E_\la(1))\mb v = \mbb K\mbb C^{-1} \cl{0}{1}\mbb E_\la(1-s)\mb f(s)ds.
\label{resk}
\end{equation}
Since the norm of $\mbb E_\la(1)$ can be made as small as one wishes by taking large $\la$, we see that $\mb v$ is uniquely defined by the Neumann series provided $\la$ is sufficiently large and hence the resolvent of $\mb A$ exists.

Let us first consider $\mbb K\geq 0$. Then the Neumann series expansion ensures that $\mb A$ {is a resolvent positive operator} and hence we can consider only $\mb f\geq 0$.  Adding together the rows in (\ref{resk})
we obtain
\begin{equation}
\sum\limits_{j=1}^{m}  v_j = \sum\limits_{j=1}^{m} \kappa_je^{-\frac{\la}{c_j}}v_j + \sum\limits_{j=1}^{m} \frac{\kappa_j}{c_j} \cl{0}{1}e^{\frac{\la}{c_j}(s-1)}f_j(s)ds,
\label{res2}
\end{equation}
where $\kappa_j = \sum\limits_{i=1}^m k_{ij}.$
Renorming $\mb X$ with the norm $\|\mb u\|_c = \sum_{j=1}^m c^{-1}_j\|u_j\|_{L_1(I)}$ and using (\ref{res1}) and (\ref{res2}), we obtain
\begin{equation}\|\mb u\|_c = \label{vest} \frac{1}{\la} \sum\limits_{j=1}^m\! v_je^{-\frac{\la}{c_j}}(\kappa_j\!-1) + \frac{1}{\la} \sum\limits_{j=1}^m\!\frac{\kappa_j\!-1}{c_j}\!\! \cl{0}{1} \!\! e^{\frac{\la}{c_j}(s-1)}f_j(s)ds + \frac{1}{\la} \|\mb f\|_c.\nn
\end{equation}
We consider three cases.\smallskip\\
 (a) $\kappa_j \leq 1$ for $j\in \mc M.$ Then $\|\mbb E_\la(-1)\mbb K\|<1$ and thus $\mb v,$ and hence $R(\la, A)$, are defined and positive for any $\la>0$. Further, dropping the  first two terms in (\ref{vest}) we get
$$
\|\mb u\|_c \leq \frac{1}{\la} \sum\limits_{j=1}^m \frac{1}{c_j}\cl{0}{1} f_j(s)ds = \frac{1}{\la}\|\mb f\|_c, \quad \la>0.
$$
Hence $(\mb A,D(\mb A))$ generates a positive semigroup of contractions in $(\mb X,\|\cdot\|_c)$.\\
(b) $\kappa_j \geq  1$ for $j\in \mc M.$ Then (\ref{vest}) implies that for some $\la>0$  and $c = 1/\la$ we have
$$
\|R(\la,\mb A)\mb f\|_c \geq  c\|\mb f\|_c
$$
and, by density of $D(\mb A)$, the application of the Arendt-Batty-Robinson theorem \cite{Ar, BaAr}, gives the existence of a positive semigroup generated by $\mb A$ in $(\mb X,\|\cdot\|_c)$. Since  $\|\cdot\|_c$ is equivalent to  $\|\cdot\|_{\mb X},$  $\mb A$ generates a positive semigroup in $\mb X$.\\
 (c) $\kappa_j < 1$ for $j\in I_1$ and $\kappa_j \geq 1$ for $j\in I_2,$ where $I_1\cap I_2 =\emptyset$ and $I_1\cup I_2 = \{1,\ldots,m\}$. Let $\mbb L = (l_{ij})_{1\leq i,j\leq m},$ where $l_{ij} = k_{ij}$ for $j\in I_2$ and $l_{ij}=1$ for $j\in I_1$.
Denoting by  $\mb A_{\mbb L}$ the operator given by the differential expression $\mathsf A$ restricted to
$
D(\mb A_{\mbb L}) = \{\mb u \in \mb W^{1}_1(I);\; \mb u(0) =  \mbb L\mb u(1)\}
$
we see, by (\ref{resk}),  that
\begin{equation}
0\leq R(\la, \mb A)\leq R(\la, \mb A_{\mbb L})
\label{resineq}
\end{equation}
for any $\la \in \rho(\mb A_{\mbb L}).$  By item (b), $\mb A_{\mbb L}$ generates a positive $C_0$-semigroup and thus satisfies the Hille--Yosida estimates. Since clearly (\ref{resineq}) yields $R^k(\la, \mb A) \leq R^k(\la, \mb A_{\mbb L})$ for any $k\in \mbb N$, for some $\omega>0$ and $M\geq 1$ we have
$$
\|   R^k(\la, \mb A)\| \leq \|   R^k(\la, \mb A_{\mbb L})\|\leq M(\la-\omega)^{-k}, \qquad \la>\omega, k\in \mbb N,
$$
and hence we obtain the generation of a semigroup by $\mb A$.

Assume now that $\mbb K$ is arbitrary. The analysis up to (\ref{resk}) remains valid. Then (\ref{resk}) can be expanded as
\begin{equation}
\mb v = \sum\limits_{n=0}^\infty (\mbb K\mbb E_\la(1))^n \mbb K\mbb C^{-1} \cl{0}{1}\mbb E_\la(1-s)\mb f(s)ds
\label{resk1}
\end{equation}
and hence
\begin{equation}
\mb u(x) = \mbb E_\la(x) \sum\limits_{n=0}^\infty (\mbb K\mbb E_\la(1))^n \mbb K\mbb C^{-1} \cl{0}{1}\mbb E_\la(1-s)\mb f(s)ds + \mbb C^{-1}\cl{0}{x}\mbb E_\la(x-s)\mb f(s)ds.
\label{resk3}
\end{equation}
Denoting now $|\mb u|=(|u_1|,\ldots,|u_n|)$ and $|\mbb K|=(|k_{ij}|)_{1\leq i,j\leq m}$ and using the fact that only $\mbb K$ may have non positive entries, we find
\begin{eqnarray*}
|\mb u(x)|&\leq & \mbb E_\la(x)\!\! \sum\limits_{n=0}^\infty (|\mbb K|\mbb E_\la(1)|)^n |\mbb K|\mbb C^{-1}\!\!\! \cl{0}{1}\mbb E_\la(1-s)\mb |\mb f|(s)ds\\
 &&\phantom{xx}+ \mbb C^{-1}\!\!\!\cl{0}{x}\!\mbb E_\la(x-s)\mb |\mb f|(s)ds = R(\la, \mb A_{|\mbb K|})|\mb f|,
\end{eqnarray*}
where $\mb A_{|\mbb K|}$ denotes the expression $\mathsf A$ restricted to $
D(\mb A_{|\mbb K|}) = \{\mb u \in \mb W^{1}_1(I);\; \mb u(0) =  \mbb |K|\mb u(1)\}
$.
So, we can write
\begin{equation}
|R(\la, \mb A_{\mbb K})\mb f|\leq R(\la, \mb A_{|\mbb K|})|\mb f|
\label{modK}
\end{equation}
and, iterating,
$$
|R(\la, \mb A_{\mbb K})^k\mb f|  \leq R(\la, \mb A_{|\mbb K|})^k|\mb f|.
$$
Using the fact that taking the modulus does not change the norm, we find
$$
\|R(\la, \mb A_{\mbb K})^k\mb f\|  \leq \|R(\la, \mb A_{|\mbb K|})^k|\mb f|\|\leq \frac{M}{\la-\omega}\|\mb f\|.
$$
with $M$ and $\omega$ following from the Hille-Yosida estimates for $A_{|\mbb K|}$.

The fact that  $\mbb K\geq 0$ yields the positivity of the semigroup follows from the first part of the proof.  To prove the converse, let $k_{ij}<0$ for some $i,j$ and consider the initial condition $\mb f(x) = (f_1(x),\ldots, f_m(x))$ with $f_k=0$ for $k\neq j$ and $f_j\in C^1([0,1])$ with $f_j(0)=f_j(1) =0,$ so that $\mb f \in D(A)$, and $f_j(x)>0$ for $0<x<1$. Then, at least for $t<\min_{1\leq j\leq m} \{1/c_j\}$,   $u_i$ satisfies
$$
\p_t u_i=c_i\p_x u_i, \quad u_i(x,0)=0, u_i(0,t) = k_{ij}f_j(1-c_jt);
$$
that is,
$$
u_i(x,t) = k_{ij}f_j\left(1+\frac{c_jx}{c_i}-c_jt\right), \qquad t\geq \frac{x}{c_i}
$$
and we see that the solution is negative for such $t$.
This ends the proof.
\end{proof}
We conclude the paper by noting an interesting corollary which also uses monotonicity properties of the problem and  which is important in asymptotic analysis, see \cite{BFN1}. For given  velocities $\mbb C = (c_1,\ldots,c_m),$ let $c = \min_{j\in \mc M}\{c_j\}$ and $\mbb C_{\mathrm{min}}= \underbrace{\mathrm\{c,\ldots,c\}}_{m\,\mathrm{times}}$. For a particular velocity matrix $\mbb C$, let $\mb A_{\mbb K,\mbb C}$ denote the generator of the semigroup solving (\ref{ACP1}).
\begin{corollary}
There holds
$$
\|e^{t\mb A_{\mbb K,\mbb C}}\| \leq  \|e^{t\mb A_{|\mbb K|, \mbb C_{\mathrm{min}}}}\|.
$$
\end{corollary}
\begin{proof}
We have by (\ref{modK}), (\ref{res1}) and (\ref{resk1}) that
$$
\|R(\la, \mb A_{\mbb K, \mbb C})\mb f\| \leq \|R(\la, \mb A_{|\mbb K|, \mbb C_{\mathrm{min}}})|\mb f|\|,
$$
hence the result follows from \cite[Theorem 1.8.3]{Pa}.
\end{proof}

\end{document}